\documentclass[12pt]{amsart}
\usepackage{latexsym, enumerate}
\usepackage{amssymb,amsmath,mathtools}

\usepackage{tikz}
\usetikzlibrary{shapes.geometric,arrows}
\tikzstyle{box} = [rectangle,text centered, draw=black]
\tikzstyle{arrow} = [thick, ->, >=stealth]

\usepackage{xcolor}

\usepackage{amsthm}

\usepackage{multirow}

\mathtoolsset{showonlyrefs}

\newcommand{\rd}{{\rm d}}
\newcommand{\e}{{\rm e}}
\newcommand{\Ran}{\mathop{\rm Ran}}
\newcommand{\N}{{\mathbb N}}
\newcommand{\R}{{\mathbb R}}
\newcommand{\C}{{\mathbb C}}
\newcommand{\Z}{{\mathbb Z}}
\newcommand\eps{\epsilon}
\newcommand{\dist}{\mathrm{dist}}
\newcommand\re{\mathrm{Re}\,}
\newcommand\im{\mathrm{Im}\,}
\newcommand\I{\mathrm{i}}
\DeclareMathOperator{\Tr}{Tr}

\DeclareMathOperator{\spec}{spec}
\newcommand\jc[1]{\textcolor{black}{#1}}         
\newtheorem{theorem}{Theorem}[section]

\newtheorem{proposition}[theorem]{Proposition}

\newtheorem{lemma}[theorem]{Lemma}
\newtheorem{remark}[theorem]{Remark}

\begin{document}

\title[Eigenvalue bounds on compact manifolds]{Eigenvalue bounds for {S}chr\"{o}dinger operators with complex potentials on compact manifolds}

%\keywords{Resolvent estimates}
 \author{Jean-Claude Cuenin}
 \address{Department of Mathematical Sciences, Loughborough University, Loughborough,
 Leicestershire, LE11 3TU United Kingdom}
 \email{J.Cuenin@lboro.ac.uk}
 
 \date{}
 \thanks{Support through the Engineering \& Physical Sciences Research Council (EP/X011488/1) is acknowledged.}

\begin{abstract}
We prove eigenvalue bounds for Schr\"odinger operator $-\Delta_g+V$ on compact manifolds with complex potentials $V$. The bounds depend only on an $L^q$-norm of the potential, and they are shown to be optimal, in a certain sense, on the round sphere and more general Zoll manifolds. These bounds are natural analogues of Frank's \cite{MR2820160} results in the Euclidean case.
\end{abstract}
\maketitle

\section{Introduction and main results}
\subsection{Spectral inclusion} Let $\Delta_g$ be the Laplace-Beltrami operator on a closed (i.e., compact boundaryless) Riemannian manifold $(M,g)$ of dimension $n\geq 2$.
We consider Schrödinger operators 
\begin{align}\label{Schrödinger operator on M}
    -\Delta_g+V \quad \mbox{on } L^2(M)
\end{align}
 with a complex-valued potential $V$. The unperturbed operator $-\Delta_g$ is a 
non-negative, unbounded self-adjoint operator in $L^2(M)$, defined with respect to the Riemannian volume measure $\rd x$. Since $-\Delta_g$ has compact resolvent, its spectrum consists of nonnegative eigenvalues $0=\lambda_0^2\leq \lambda_1^2\leq \ldots$ with $\lambda_k^2\to\infty$, and the set of corresponding eigenfunctions $(e_j)$ forms an orthonormal basis of $L^2(M)$. 
By standard perturbation theory, if $V\in L^{\infty}(M)$, then the eigenvalues of $-\Delta_g+V$ are contained in a union of closed disks, 
\begin{align}\label{Linfty bound}
\spec(-\Delta_g+V)\subset \bigcup_{k=0}^{\infty}D(\lambda_k^2,r), \quad r=\|V\|_{L^{\infty}(M)}.
\end{align}
 In this paper, we deal with unbounded potentials, i.e., $V\in L^q(M)$ for some $q<\infty$. 

  \begin{theorem}\label{theorem main}
Let $\Delta_g$ be the Laplace-Beltrami operator on a closed Riemannian manifold $(M,g)$ of dimension $n\geq 2$, and let $q> n/2$. Then there exists a constant $C=C(M,g,q)$ such that for all $V\in L^q(M)$,
\begin{align}\label{spectral inclusion}
    \spec(-\Delta_{g}+V)\subset \bigcup_{k=0}^{\infty}D(\lambda_k^2,Cr_k)\cup \{z\in\C: |z|^{\frac{1}{2}}(1+|z|)^{-\sigma(q)}\leq C \|V\|_{L^q(M)}\},
\end{align}
where $\lambda_k^2$ are the eigenvalues of $-\Delta_g$, 
$r_k:=\|V\|_{L^q(M)}(1+\lambda_k)^{2\sigma(q)}$,
and
\begin{align}\label{sigma(q)}
    \sigma(q):=\begin{cases}
        \frac{n}{2q}-\frac{1}{2}\quad &\mbox{if } \frac{n}{2}\leq  q\leq \frac{n+1}{2},\\
        \frac{n-1}{4q}\quad &\mbox{if } \frac{n+1}{2}\leq q\leq \infty. 
    \end{cases}   
\end{align}
If $n\geq 3$ and $q=n/2$, then there exist constants $C,c>0$, depending on $M,g,q$, such that for all $V\in L^{\frac{n}{2}}(M)$ satisfying $\|V\|_{L^{\frac{n}{2}}(M)}\leq c$, we have
\begin{align}\label{q=n/2 spectral inclusion}
    \spec(-\Delta_{g}+V)\subset \bigcup_{k=0}^{\infty}D(\lambda_k^2,Cr_k),\quad r_k:=\|V\|_{L^{\frac{n}{2}}(M)}(1+\lambda_k).
\end{align}
\end{theorem}

We emphasize that the constant in \eqref{spectral inclusion} is \textit{independent of $V$}. In particular, $V$ is allowed to be \emph{frequency-dependent}; cf. Theorem \ref{theorem optimality sphere} below. Estimates of a similar type were established in \cite{MR3655948} for the Landau Hamiltonian (that is, the Schr\"odinger operator with constant magnetic field) in $\R^n$ ($n$ even). In contrast to the situation here, the radii of the disks tend to zero (i.e., the powers of $\lambda_k$ are negative) in this case.

Under the assumptions of Theorem \ref{theorem main}, the operator $-\Delta_g+V$ can be defined as an $m$-sectorial operator by standard quadratic form methods; see, e.g., \cite{MR1335452}, or more specifically, \cite[Appendix B]{MR3717979}. The proof uses Sobolev embedding, as in the case of real-valued $V$, considered in \cite[Section 6]{MR4445914}.

\begin{remark}
\noindent (i) We will refer to the union of disks $\bigcup_{k=0}^{\infty}D(\lambda_k^2,Cr_k)$ as ``the first set" in \eqref{spectral inclusion} and to $\{z\in\C: |z|^{\frac{1}{2}}(1+|z|)^{-\sigma(q)}\leq C \|V\|_{L^q(M)}\}$ as ``the second set". By distinguishing the cases $|z|\leq 1$ or $|z|>1$, it is easy to see that the second set is contained in the disk
\begin{align}\label{second set included in disk at origin}
 \{z\in\C: |z|\leq \max((2^{\sigma(q)}C \|V\|_{L^q(M)})^2,(2^{\sigma(q)}C \|V\|_{L^q(M)})^{\frac{1}{1/2-\sigma(q)}})\} 
\end{align}
centered at the origin. Since $\lambda_0=0$, the disk $D(\lambda_0^2,Cr_0)$ is also centered at the origin and could be absorbed in \eqref{second set included in disk at origin} by including the term $C\|V\|_{L^q(M)}$ in the maximum. However, the decomposition \eqref{spectral inclusion} arises naturally in the proof.

\noindent (ii) Asymptotically, as $k\to\infty$, the disks in the first set are contained inside the region
\begin{align}\label{asymptotic region main theorem}
    \{z\in\C: |\im z|\leq C\|V\|_{L^q(M)} (\re z)^{\sigma(q)}\}.
\end{align}
Note that $\sigma(q)<1/2$ for $q>n/2$ and $\sigma(q)=1/2$ for $q=n/2$. Hence, for the latter case, the region \eqref{asymptotic region main theorem} is the interior of a parabola around the positive $\re z$ axis, while for $q>n/2$, it is much smaller, approaching a strip $|\im z|\leq C\|V\|_{L^q(M)}$ as $q\to\infty$. 
\end{remark}

\subsection{Comparison with the Euclidean case}
 We recall the bounds for the Schr\"odinger operators on $L^2(\R^n)$ that we wish to compare our results to: For $n/2<q\leq (n+1)/2$, 
\begin{align}\label{Rn bound}
    \spec(-\Delta+V)\subset [0,\infty)\cup \Big\{z\in\C:|z|^{q-\frac{n}{2}}\leq C_{n,q}\int_{\R^n}|V(x)|^{q}\rd x\Big\}.
\end{align}      
Moreover, if $n\geq 3$ and $q=n/2$, then there exists a constant $c_{n}$ such that there are no eigenvalues outside $[0,\infty)$ if $\|V\|_{L^{n/2}(\R^n)}\leq c_n$.
The one-dimensional bound is due to Abramov--Aslanyan--Davies \cite{MR1819914}, and the higher dimensional bounds are due to Frank \cite{MR2820160}. It is known \cite{MR4561804} that the inequality \eqref{Rn bound} fails for $q>(n+1)/2$, whereas Laptev and Safronov \cite{MR2540070} had conjectured its validity up to $q<n$. The limitation to $q\leq (n+1)/2$ essentially comes from the Knapp example in harmonic analysis. It requires that $V$ decay sufficiently fast at infinity in an average sense. On a compact manifold, such a condition is of course meaningless. 
However, if we consider the $n$-dimensional flat torus $\mathbb{T}^n=\R^n/\Z^n$ and rescale by a factor $L>0$, we obtain
\begin{align}\label{spectra scaled}
    \spec_{L^2(L\mathbb{T}^n)}(-\Delta+L^{-2}V(\cdot/L))=L^{-2}\spec_{L^2(\mathbb{T}^n)}(-\Delta+V).
\end{align} 
Hence, for the rescaled operator, the decay of the potential is now measured relative to the scale~$L$. According to Theorem \ref{theorem main}, its spectrum is contained in the $L^{-2}$ dilation of the set on the right-hand side of \eqref{spectral inclusion}. In Section \ref{Section Scaling considerations}, we will show that, as $L\to\infty$, all limit points of this set lie in the right-hand side of \eqref{Rn bound}, provided that $n/2<q\leq (n+1)/2$. More precisely, it follows from the proof of Lemma \ref{lemma scaling limit} that limit points of the first set in \eqref{spectral inclusion} (the union of disks) are contained in $[0,\infty)$, while limit points of the second set of \eqref{spectral inclusion} are contained in the second set on the right-hand side of \eqref{Rn bound}. 

We could also use \eqref{spectral inclusion} to prove \eqref{Rn bound}, but since the latter is already known, we will not pursue this direction. More importantly, the argument shows that since \eqref{Rn bound} is known to be optimal~\cite{MR4561804}, the second set on the right-hand side of \eqref{spectral inclusion} is also optimal in the sense that, for large $|z|$, the power $|z|^{\frac{1}{2}-\sigma(q)}$ cannot be improved on a torus, in the regime where the sidelength $L$ tends to infinity.

\subsection{Saturation on the sphere and Zoll manifolds}
We turn to the optimality of the first set on the right-hand side of~\eqref{spectral inclusion} when $M =S^n$, the round sphere, or more generally, a Zoll manifold. These are manifolds whose geodesics are periodic with a common minimal period. Without loss of generality after rescaling the metric, we may assume that
the common minimal period is $2\pi$. By a theorem of Weinstein
\cite{MR482878}, there is a constant $\alpha=\alpha_M\geq 0$
so that all the nonzero eigenvalues of $-\Delta_g$ cluster around the values $(k+\alpha)^2$, $k\in\N$. More precisely, 
there is a constant $A=A_M$ such that each non-zero eigenvalue is contained in a cluster $[(k+\alpha-A/k)^2,(k+\alpha+A/k)^2]$ for some $k\in\N$, and the number of eigenvalues in this cluster, counted with multiplicity, is $d_k\approx k^{n-1}$. We label the eigenvalues in the $k$-th cluster by $\lambda_{k,j}^2$, $j=1,\ldots,d_k$. 

In the case of the sphere $S^n$, the eigenvalues of the Laplacian are given explicitly by $\lambda_{k,j}^2 = k(k+n-1)$, and thus the first set in~\eqref{spectral inclusion} is a union of disks of the form
\begin{align}\label{D_k sphere}
D_k := D(k(k+n-1),Ck^{2\sigma(q)} \|V\|_{L^q(S^n)}),
\end{align}
when $k\geq 1$. We will consider the asymptotic regime $k\to\infty$ here and allow $V$ to depend on $k$. For general Zoll manifolds, we will additionally assume that 
\begin{align}\label{assumption V Zoll}
k^{2\sigma(q)} \|V\|_{L^q(M)}\gtrsim 1,   
\end{align}
which ensures that all the disks $D(\lambda_{k,j}^2,Cr_{k,j})$ corresponding to $k$-th cluster are contained in the larger disk
\begin{align}\label{D_k tilde Zoll}
    \tilde{D}_k := D((k+\alpha)^2,Ck^{2\sigma(q)} \|V\|_{L^q(S^n)}),
\end{align}
for some larger constant $C$, depending on $A$ and the lower bound in \eqref{assumption V Zoll}.

\begin{theorem}\label{theorem optimality sphere}
i) Let $n\geq 2$ and $q>n/2$. 
For any $\rho>0$, $\theta\in [0,2\pi)$ and \jc{$k\gg 1+\rho^{1/(1-2\sigma(q))}$},
there is a potential $V=V_{k,\theta,\rho} \in L^{q}(S^n)$, $\|V\|_{L^{q}(S^n)}\approx \rho$, such that
    \begin{align}\label{optimality on the sphere}
    k(k+n-1)+\rho\e^{\I\theta}(1+\mathcal{O}(\rho k^{2\sigma(q)-1})) k^{2\sigma(q)}\in \spec(-\Delta_{S^n}+V)\cap D_{k}. 
\end{align}
If $n\geq 3$ and $q=n/2$, then \eqref{optimality on the sphere} holds provided $\rho\ll 1$.\\

\noindent ii) If $M$ is a general Zoll manifold of dimension $n\geq 2$, $q>n/2$, $\rho>0$, $\theta\in [0,2\pi)$ and $k\gg 1$, and if we assume in addition that $k^{2\sigma(q)}\rho\gg 1$, then there is a potential $V=V_{k,\theta,\rho} \in L^{q}(M)$, $\|V\|_{L^{q}(S^n)}\approx \rho$, such that
    \begin{align}\label{optimality Zoll}
    (k+\alpha)^2+\rho\e^{\I\theta}(1+\mathcal{O}(\rho k^{2\sigma(q)-1})) k^{2\sigma(q)}\in \spec(-\Delta_g+V)\cap \tilde{D}_{k}. 
\end{align}
If $n\geq 3$ and $q=n/2$, then \eqref{optimality Zoll} holds provided $\rho\ll 1$.
\end{theorem}

\begin{remark}\label{remark theorem optimality}
 (i) The case $q=\infty$ is included but is not particularly interesting, as it is easily seen that \eqref{Linfty bound} is sharp when  $V$ is taken to be a constant potential.\\

 \noindent (ii) Theorem \ref{theorem optimality sphere} shows that, on the sphere or more general Zoll manifolds, the radius of the disks in Theorem \ref{theorem main}, $r_k=\|V\|_{L^q(M)}(1+\lambda_k)^{2\sigma(q)}$, cannot be decreased (up to the value of the constant). More precisely, \eqref{spectral inclusion} cannot hold if $r_k$ there is replaced by $\eps(\lambda)r_k$, where $\eps:\R_+\to\R_+$ is any decreasing function satisfying $\lim_{\lambda\to\infty}\eps(\lambda)=0$. 
\end{remark}

\subsection{Resolvent estimates}\label{Subsection intro: resolvent estimates}
Frank \cite{MR2820160} proved the Euclidean bound \eqref{Rn bound} by combining uniform $L^p\to L^{p'}$ resolvent bounds, established by Kenig, Ruiz, and Sogge \cite{MR894584} in 1987, with the well-known Birman--Schwinger principle. The latter stipulates that $z$ is an eigenvalue of $-\Delta+V$ if and only if $-1$ is an eigenvalue of the compact operator $|V|^{1/2}(-\Delta-z)^{-1}V^{1/2}$. Here and in the following, $V^{1/2}:=V/|V|^{1/2}$. We will follow the same strategy for proving Theorem~\ref{theorem main}. Compared to the Euclidean case, $L^p\to L^{p'}$ resolvent bounds for the Laplace--Beltrami operator on compact manifolds have a comparatively short history. To make a direct connection to Theorem \ref{theorem main} we shall adopt the following convention throughout this paper,
\begin{align}\label{1/q=1/p-1/p'}
     \frac{1}{q}=\frac{1}{p}-\frac{1}{p'}.
\end{align}
Dos Santos Ferreira, Kenig, and Salo \cite{MR3200351} proved the uniform bound 
\begin{align}\label{Dos Santos Ferreira, Kenig and Salo}
\|(-\Delta_g-z)^{-1}\|_{L^{p}(M)\to L^{p'}(M)}\lesssim|z|^{\sigma(q)-\frac{1}{2}},\quad |\im\sqrt{z}|\geq 1,
\end{align}
for $p=2n/(n+2)$ and $n\geq 3$, in which case $\sigma(q)-\frac{1}{2}=0$. The latter improves earlier work of Shen \cite{MR2366961} for the flat torus. Here and in the following, we are using the principal branch of the square root on $\C\setminus (-\infty,0]$ with $\re\sqrt{z}\geq 0$. Frank and Schimmer \cite{MR3620715}, and, independently, Burq, Dos Santos Ferreira, and Krupchyk \cite{MR3848231}, proved \eqref{Dos Santos Ferreira, Kenig and Salo} for the endpoint $p=2(n+1)/(n+3)$ and $n\geq 2$. A straightforward interpolation argument, combined with some additional considerations specific to the case  $n=2$, will yield the estimates for the intermediate values of $p$ in the region where $|\im \sqrt{z}| \geq 1$ (see also \cite[Remark 9]{MR3620715}). Note that the estimates are uniform in this region, as $\sigma(q)-\frac{1}{2}\leq 0$. We will see that these uniform resolvent estimates account for the second set on the right-hand side of~\eqref{spectral inclusion}. In contrast, the first set, consisting of a union of disks, corresponds to non-uniform resolvent estimates in the complementary region:
\begin{align}\label{non-uniform resolvent bounds}
  \|(-\Delta_g-z)^{-1}\|_{L^{p}(M)\to L^{p'}(M)}\lesssim\dist(z,\spec(-\Delta_g))^{-1}(1+ |z|)^{\sigma(q)},\quad &|\im\sqrt{z}|< 1.
\end{align}
These estimates are new in this generality. They are proved by using Sogge's spectral cluster bounds \cite{MR930395}, in the spirit of Bourgain, Shao, Sogge, and Yao~\cite{MR3302640} and the author's own work~\cite{MR4664427}. The “universal” bounds of Sogge \cite{MR930395}, valid for any closed manifold $M$, state that
\begin{align}\label{universal cluster bound Sogge}
    \|\mathbf{1}(\sqrt{-\Delta_g}\in [\lambda-1/2,\lambda+1/2])\|_{L^{2}(M)\to L^{p'}(M)}\lesssim(1+\lambda)^{\sigma(q)}.
\end{align}
Sogge's bounds hold for the larger range $2\leq p'\leq \infty$ (see \eqref{def. nu(p')}, \eqref{Sogge 1}) than indicated by \eqref{sigma(q)}, \eqref{1/q=1/p-1/p'}. However, when using \eqref{universal cluster bound Sogge} to prove resolvent estimates, the latter restrictions are imposed by Sobolev embedding.

We remark that, in the special case $M=S^n$, $n\geq 3$ and $p=2n/(n+2)$, \eqref{non-uniform resolvent bounds} was proved in \cite[Theorem 1.1]{MR3302640}.
In fact, they showed that \eqref{non-uniform resolvent bounds} is an approximate equality in this case. This shows that the region of $z\in\C$ for which \eqref{Dos Santos Ferreira, Kenig and Salo} holds for $n\geq 3$ and $p=2n/(n+2)$, is essentially optimal on the sphere and, more generally, on any Zoll manifold.
Similarly, optimality of the region $|\im\sqrt{z}|>1$ in \eqref{Dos Santos Ferreira, Kenig and Salo} for Zoll manifolds, in the case $p=2(n+1)/(n+3)$ and $n\geq 2$, was proved in \cite{MR3620715}.

\subsection*{Notation} 
We write $A\lesssim B$ for two non-negative quantities $A,B\geq 0$ to indicate
that there is a constant $C>0$ such that $A\leq C B$. If $C$ depends on a list of parameters $\ell$, we sometimes emphasize this by writing $A\lesssim_{\ell} B$. However, the dependence of $C$ on fixed parameters like the manifold $(M,g)$, the dimension $n$ and the Lebesgue exponent $q$ is usually omitted. The notation $A\approx B$ means $A\lesssim B\lesssim A$. We will often consider the frequency $\lambda_k$ as an asymptotic parameter, $\lambda_k\to\infty$ (thus also $k\to \infty$). The notation $A\ll B$ means that $A\leq c B$, where $c>0$ is suitably small but fixed (in particular, independent of $k$). Sometimes, this notation is used to shorten a ``there exists..." statement, e.g.\ in Theorem \ref{theorem optimality sphere} $k\gg 1$ means ``there exists $k_0\in \N$ such that for all $k\geq k_0$, we have...". 
We denote the $L^p$ norm of a function $f$ on $M$, with respect to the Riemannian measure $\rd x$, by $\|f\|_{L^p(M)}$. If $T:L^s(M)\to L^r(M)$ is a bounded linear operator, we denote its operator norm by $\|T\|_{L^s(M)\to L^r(M)}$. If $r=s=2$, we sometimes write $\|T\|:=\|T\|_{L^2(M)\to L^2(M)}$.
For a subset $\Omega\subset \R$, we define $\mathbf{1}(\sqrt{-\Delta_g}\in \Omega)$ by functional calculus. Explicitly, if we denote the orthogonal projection onto the eigenfunction $e_j$ by $E_j$, that is $E_jf=\langle e_j,f\rangle e_j$, then $\mathbf{1}(\sqrt{-\Delta_g}\in \Omega)f:=\sum_{\lambda_j\in\Omega}E_jf$. The inner product is defined as $\langle u,v\rangle:=\int_M \overline{u}v\,\rd x$. 
The closed disk centered at $r\in\C$, with radius $r>0$, is denoted by $D(z,r)\subset \C$. 

\section{Resolvent estimates}\label{Section Resolvent estimates}
In the following, we will use the abbreviation $$d(z):=\dist(z,\spec(-\Delta_g)).$$
Since $\Delta_g$ is self-adjoint, we have 
\begin{align}\label{trivial L2-L2 bound}
    \|(-\Delta_g-z)^{-1}\|_{L^{2}(M)\to L^{2}(M)}= d(z)^{-1}.
\end{align}
Note that $p=2$ corresponds to $q=\infty$ and $\sigma(\infty)=0$. Since $q$ will not play any role in this section (except in the proof of Theorem \ref{theorem main}), we will use the notation $\nu(p'):=\sigma(q)$ when $1/q=1/p-1/p'$. Explicitly,
\begin{align}\label{def. nu(p')}
    \nu(p')=\begin{cases}
        \frac{n-1}{2}(\frac{1}{2}-\frac{1}{p'})\quad &\mbox{if } 2\leq p'\leq \frac{2(n+1)}{n-1},\\
        n(\frac{1}{2}-\frac{1}{p'})-\frac{1}{2}\quad &\mbox{if } \frac{2(n+1)}{n-1}\leq  p'\leq \infty.
    \end{cases}   
\end{align}

\begin{proposition}\label{prop. Lp to Lp' resolvent estimate}
    Let $\Delta_g$ be the Laplace--Beltrami operator on an $n$-dimensional closed Riemannian manifold $(M,g)$. Let $2n/(n+2)\leq p\leq 2$ if $n\geq 3$ and $1<p\leq 2$ if $n=2$.
    Then there is a constant $C=C(M,g,p)$ such that for all $z\in \C\setminus\spec(-\Delta_g)$,
\begin{align}
\|(-\Delta_g-z)^{-1}\|_{L^{p}(M)\to L^{p'}(M)}
&\leq C |z|^{\nu(p')-\frac{1}{2}},\quad |\im\sqrt{z}|\geq 1,\label{universal resolvent bound imsqrtz>1}\\
\|(-\Delta_g-z)^{-1}\|_{L^{p}(M)\to L^{p'}(M)}
&\leq C d(z)^{-1}(1+ |z|)^{\nu(p')},\quad |\im\sqrt{z}|< 1.\label{universal resolvent bound imsqrtz<1}
\end{align}  
\end{proposition}

\begin{remark}
(i) The region $|\im\sqrt{z}|\geq 1$ corresponds to the set
 \begin{align}\label{Xi}
     \Xi:=\{z\in\C:(\im z)^2\geq 4(\re z+1)\},
 \end{align}
which is the exterior of a parabola.\\

\noindent (ii) The new result here is \eqref{universal resolvent bound imsqrtz<1}, whereas \eqref{universal resolvent bound imsqrtz>1} is essentially known. 
\end{remark}

\begin{proof}[Proof of \eqref{universal resolvent bound imsqrtz>1}]
    If $p=2n/(n+2)$ and $n\geq 3$, or if $2n/(n+2)< p\leq 2(n+1)/(n+3)$ and $n\geq 2$, then \eqref{universal resolvent bound imsqrtz>1} follows from \cite{MR3200351}, or from \cite[Theorem 1 and Remark 9]{MR3620715}, respectively. For $\re z\leq 0$, we have $d(z)=|z|$, so \eqref{trivial L2-L2 bound} is stronger than \eqref{universal resolvent bound imsqrtz>1} when $p=2$. For $\re z\geq 0$, $\im\sqrt{z}\geq 1$, we have $|z|\lesssim d(z)^2$, hence \eqref{trivial L2-L2 bound} implies \eqref{universal resolvent bound imsqrtz>1} in this region as well when $p=2$. The Riesz–Thorin interpolation theorem between $p=2n/(n+2)$, $p=2(n+1)/(n+3)$ and $p=2$ then yields \eqref{universal resolvent bound imsqrtz>1} for $2n/(n+2)\leq p\leq 2$ and $n\geq 3$. For $n=2$, one interpolates with some $p>1$ arbitrarily close to $p=1$.
\end{proof}
 
 Before proceeding with the proof of Proposition \ref{prop. Lp to Lp' resolvent estimate}, we record the following simple lemma.

\begin{lemma}\label{lemma |z|<A}
For any $A\geq 1$ there exists a constant $C_A$ (also depending on $n,q$) such that
    \begin{align}
    \|(-\Delta_g-z)^{-1}\|_{L^{p}(M)\to L^{p'}(M)}\leq C_Ad(z)^{-1},\quad \re z\leq A.
\end{align}
\end{lemma}

\begin{proof}
    We employ a spectral decomposition into frequencies $\lambda_j\leq \sqrt{2A}$ and $\lambda_j>\sqrt{2A}$. For the latter, we have a uniform bound
\begin{align}
\|\mathbf{1}(-\Delta_g>2A)(-\Delta_g-z)^{-1}\|_{L^{p}(M)\to L^{p'}(M)}\leq C_A,\quad \re z\leq A.
\end{align}
Indeed, since for $\re z\leq A$,
\begin{align}
    \|\mathbf{1}(-\Delta_g>2A)(-\Delta_g+1)(-\Delta_g-z)^{-1}\|_{L^2(M)\to L^2(M)}\leq \sup_{\tau>2A}\frac{\tau+1}{|\tau-z|}\lesssim 1,
\end{align}
it follows from Sobolev embedding (and its dual)
that if $f$ has frequencies $\lambda_j>\sqrt{2A}$ (i.e., $f$ is in the range of $\mathbf{1}(-\Delta_g>2A)$), then
\begin{align}
 \|f\|_{L^{p'}(M)}
 &\lesssim \|(-\Delta_g+1)^{\frac{n}{2}(\frac{1}{2}-\frac{1}{p'})}f\|_{L^{2}(M)}\\
 &\lesssim \|(-\Delta_g+1)^{\frac{n}{2}(\frac{1}{2}-\frac{1}{p'})-1}(-\Delta_g-z)f\|_{L^{2}(M)}\\
 &\lesssim \|(-\Delta_g-z)f\|_{L^{p}(M)}.
\end{align}
Next, if $f$ has frequencies $\lambda_j\leq \sqrt{2A}$, then
by Young's inequality, $\|f\|_{L^r(M)}\lesssim_A \|f\|_{L^s(M)}$ for all $1\leq s\leq r\leq \infty$ (see e.g.\ \cite[Lemma 3.1]{MR3302640}). In particular, if $(e_j)$ is an orthonormal basis of eigenfunctions of $-\Delta_g$ in $L^2(M)$ with eigenvalues $\lambda_j^2$ and $E_j$ denotes the orthogonal projection $E_jf=\langle e_j,f\rangle e_j$, then
\begin{align}
    \|E_j f\|_{L^{p'}(M)}\leq \|e_j\|_{L^{p'}(M)}^2 \|f\|_{L^{p}(M)}\lesssim \|f\|_{L^{p}(M)}
\end{align}
for $\lambda_j\leq \sqrt{2A}$ and $E_jf=0$ for $\lambda_j>\sqrt{2A}$.
Thus,
\begin{align}
 \|(-\Delta_g-z)^{-1}f\|_{L^{p'}(M)}\leq \sum_{\lambda_j\leq \sqrt{2A}}\frac{1}{|\lambda_j^2-z|}\|E_j f\|_{L^{p'}(M)}\lesssim_A d(z)^{-1}   \|f\|_{L^{p}(M)}.
\end{align}
In the last inequality, we also used that $\#\{j:\lambda_j\leq \sqrt{2A}\}\lesssim_A 1$ since the spectrum is discrete. 
\end{proof}

 We now turn to the proof of \eqref{universal resolvent bound imsqrtz<1}.

\begin{proof}[Proof of \eqref{universal resolvent bound imsqrtz<1}] If $\re \sqrt{z}\leq 10$, then \eqref{universal resolvent bound imsqrtz<1} follows from Lemma \ref{lemma |z|<A}, so it remains to consider the case $\re\sqrt{z}>10$. 
We recall that we are also assuming $|\im\sqrt{z}|<1$.
We will split the resolvent into two pieces, one involving frequencies close to $\re\sqrt{z}$ and the other involving frequencies at least a unit distance from $\re\sqrt{z}$, i.e.\ $|\lambda_j-\re\sqrt{z}|\leq \delta$ for the first piece and $|\lambda_j-\re\sqrt{z}|> \delta$ for the second piece. Here, $0<\delta\leq 1$ is fixed, and we will take $\delta=1/2$ throughout (any other choice can be reduced to this one by scaling the metric). 
We will prove 
\begin{align}
\|\mathbf{1}(\sqrt{-\Delta_g}\in [\re\sqrt{z}-1/2,\re\sqrt{z}+1/2])(-\Delta_g-z)^{-1}\|_{L^{p}(M)\to L^{p'}(M)}\lesssim d(z)^{-1}|z|^{\nu(p')},\label{frequencies close to re sqrt z}\\
\|\mathbf{1}(\sqrt{-\Delta_g}\notin [\re\sqrt{z}-1/2,\re\sqrt{z}+1/2])(-\Delta_g-z)^{-1}\|_{L^{p}(M)\to L^{p'}(M)}\lesssim|z|^{\nu(p')-\frac{1}{2}}.\label{frequencies far from re sqrt z}
\end{align}
Combining these estimates yields the claimed bound \eqref{universal resolvent bound imsqrtz<1}. Note that \eqref{frequencies far from re sqrt z} is stronger since $|z|^{-1/2}\lesssim d(z)^{-1}$ for $\im\sqrt{z}<1$, $\re\sqrt{z}>10$.\\

The proof of \eqref{frequencies close to re sqrt z} uses Sogge's spectral cluster bounds \cite{MR930395},
\begin{align}
\|\mathbf{1}(\sqrt{-\Delta_g}\in [\lambda-1/2,\lambda+1/2])\|_{L^{2}(M)\to L^{p'}(M)}\lesssim(1+\lambda)^{\nu(p')},\label{Sogge 1}\\
\|\mathbf{1}(\sqrt{-\Delta_g}\in [\lambda-1/2,\lambda+1/2])\|_{L^{p}(M)\to L^{2}(M)}\lesssim(1+\lambda)^{\nu(p')},\label{Sogge 2}
\end{align}
with $\lambda=\re\sqrt{z}$.
Thus, by \eqref{Sogge 1}, \eqref{Sogge 2}, \eqref{trivial L2-L2 bound},
\begin{align}\label{proof of frequencies close to re sqrt z}
 &\|\mathbf{1}(\sqrt{-\Delta_g}\in [\re\sqrt{z}-1/2,\re\sqrt{z}+1/2])(-\Delta_g-z)^{-1}f\|_{L^{p'}(M)}\\
 &\lesssim (\re\sqrt{z})^{\nu(p')}\|\mathbf{1}(\sqrt{-\Delta_g}\in [\re\sqrt{z}-1/2,\re\sqrt{z}+1/2])(-\Delta_g-z)^{-1}f\|_{L^{2}(M)}\\
 &\leq d(z)^{-1} (\re\sqrt{z})^{\nu(p')}\|\mathbf{1}(\sqrt{-\Delta_g}\in [\re\sqrt{z}-1/2,\re\sqrt{z}+1/2])f\|_{L^{2}(M)}\\
 &\lesssim  d(z)^{-1}(\re\sqrt{z})^{2\nu(p')}\|f\|_{L^{p}(M)},
\end{align}
which proves \eqref{frequencies close to re sqrt z} since $\re\sqrt{z}\leq |z|^{1/2}$. A more general version of the above argument can be found in \cite[Lemma 2.3]{MR3302640} and \cite[Lemma 3.1]{MR4664427}.\\

The proof of \eqref{frequencies far from re sqrt z} is reduced to the case $|\im\sqrt{z}|\geq 1$ by the following argument (similar to the proof of \cite[(2.33)]{MR3302640}: Choose $z_1\in\C$ with $\re\sqrt{z_1}=\re\sqrt{z}$ and $|\im\sqrt{z_1}|=1$. Then
\begin{align}
& \|\mathbf{1}(\sqrt{-\Delta_g}\notin [\re\sqrt{z}-1/2,\re\sqrt{z}+1/2])(-\Delta_g-z)^{-1}f\|_{L^{p'}(M)}\\
&\leq \|\mathbf{1}(\sqrt{-\Delta_g}\notin [\re\sqrt{z}-1/2,\re\sqrt{z}+1/2])(-\Delta_g-z_1)^{-1}f\|_{L^{p'}(M)}\\
&+\|\mathbf{1}(\sqrt{-\Delta_g}\notin [\re\sqrt{z}-1/2,\re\sqrt{z}+1/2])[(-\Delta_g-z)^{-1}-(-\Delta_g-z_1)^{-1}]f\|_{L^{p'}(M)}=I+II.
\end{align}
For the first term $I$, we have 
\begin{align}
I&\leq \|(-\Delta_g-z_1)^{-1}f\|_{L^{p'}(M)}+\|\mathbf{1}(\sqrt{-\Delta_g}\in [\re\sqrt{z}-1/2,\re\sqrt{z}+1/2])(-\Delta_g-z_1)^{-1}f\|_{L^{p'}(M)}\\
&=\mathcal{I}+\mathcal{II}.
\end{align}
Since $\im\sqrt{z_1}=1$, it follows from \eqref{universal resolvent bound imsqrtz>1} that $\mathcal{I}$ is bounded by the right-hand side of \eqref{frequencies far from re sqrt z}. For $II$ and $\mathcal{II}$, we will again use Sogge's bounds \eqref{Sogge 1}, \eqref{Sogge 2} to show that
\begin{align}\label{II+mathcal II}
    II+\mathcal{II}\lesssim |z|^{\nu(p')-\frac{1}{2}}\|f\|_{L^{p}(M)}.
\end{align}
The proof of the bound for $\mathcal{II}$ is similar to that of \eqref{frequencies close to re sqrt z}. The only difference is that the $L^2$ bound in the third line of \eqref{proof of frequencies close to re sqrt z} is replaced by 
\begin{align}
    \|\mathbf{1}(\sqrt{-\Delta_g}\in [\re\sqrt{z}-1/2,\re\sqrt{z}+1/2])(-\Delta_g-z_1)^{-1}f\|_{L^{2}(M)}\lesssim |z|^{-\frac{1}{2}}\|f\|_{L^{2}(M)},
\end{align}
where we used that, for $\im\sqrt{z_1}=1$, $\re\sqrt{z_1}>10$,
\begin{align}
    |\tau^2-z_1|=|\tau-\sqrt{z_1}||\tau+\sqrt{z_1}|\geq (\im\sqrt{z_1})(\re\sqrt{z_1})\geq \re\sqrt{z}\gtrsim |z|^{\frac{1}{2}}.
\end{align}
Given the resolvent identity
\begin{align}
(-\Delta_g-z)^{-1}-(-\Delta_g-z_1)^{-1}=(z-z_1)(-\Delta_g-z)^{-1}(-\Delta_g-z_1)^{-1}  
\end{align}
and the fact that $|z-z_1|\lesssim\re\sqrt{z}$, the bound for $II$ in \eqref{II+mathcal II} would follow from
\begin{align}
 \|\mathbf{1}(\sqrt{-\Delta_g}\notin [\re\sqrt{z}-1/2,\re\sqrt{z}+1/2])(-\Delta_g-z)^{-1}(-\Delta_g-z_1)^{-1}\|_{L^{p}(M)\to L^{p'}(M)}\lesssim |z|^{\nu(p')-1}. 
\end{align}
Since the left-hand side is dominated by 
\begin{align}
    \|\mathbf{1}(\sqrt{-\Delta_g}\notin [\re\sqrt{z}-1/2,\re\sqrt{z}+1/2])(-\Delta_g-z)^{-1}\|_{L^{2}(M)\to L^{p'}(M)}\|(-\Delta_g-z_1)^{-1}\|_{L^{p}(M)\to L^{2}(M)},
\end{align}
and since the $L^{2}(M)\to L^{p'}(M)$ norm equals the $L^{p}(M)\to L^{2}(M)$ norm of the adjoint operator, it suffices to prove
\begin{align}
 \|\mathbf{1}(\sqrt{-\Delta_g}\notin [\re\sqrt{z}-1/2,\re\sqrt{z}+1/2])(-\Delta_g-\overline{z})^{-1}\|_{L^{p}(M)\to L^{2}(M)}&\lesssim |z|^{\frac{\nu(p')}{2}-\frac{1}{2}},\label{LpL2 1}\\
 \|(-\Delta_g-z_1)^{-1}\|_{L^{p}(M)\to L^{2}(M)}&\lesssim |z|^{\frac{\nu(p')}{2}-\frac{1}{2}}\label{LpL2 2}.
\end{align}
We start with the proof of \eqref{LpL2 2}. Since $\im\sqrt{z_1}=1$ and $\re\sqrt{z_1}=\re\sqrt{z}$, we have the lower bound
\begin{align}
    |\lambda_j^2-z_1|=|\lambda_j+\sqrt{z_1}||\lambda_j-\sqrt{z_1}|\gtrsim (\re\sqrt{z})(|\lambda_j-\re\sqrt{z}|+1),
\end{align}
Thus, by $L^2$ orthogonality,
\begin{align}
\|(-\Delta_g-z_1)^{-1}f\|_{L^{2}(M)}^2&=\sum_{j=0}^{\infty}|\lambda_j^2-z_1|^{-2}\|E_jf\|_{L^{2}(M)}^2\\    
&\lesssim(\re\sqrt{z})^{-2}\sum_{j=0}^{\infty}(|\lambda_j-\re\sqrt{z}|+1)^{-2}\|E_jf\|_{L^{2}(M)}^2\\
&\lesssim(\re\sqrt{z})^{-2}\sum_{k=1}^{\infty}\sup_{\tau\in[k-1,k)}(|\tau-\re\sqrt{z}|+1)^{-2}\|\sum_{\lambda_j\in [k-1,k)}E_jf\|_{L^{2}(M)}^2\\
&\lesssim (\re\sqrt{z})^{-2}\|f\|_{L^{p}(M)}^2\sum_{k=1}^{\infty}(|k-\re\sqrt{z}|+1)^{-2}k^{2\nu(p')}.
\end{align}
In the last line, we used Sogge's spectral cluster bound \eqref{Sogge 2}. 
We split the sum into two pieces, one with $|k-\re\sqrt{z}|\leq \frac{1}{2}\re\sqrt{z}$ and the other with $|k-\re\sqrt{z}|> \frac{1}{2}\re\sqrt{z}$. For the first piece, we use $k\leq \frac{3}{2}\re\sqrt{z}$ to estimate
\begin{align}
    \sum_{|k-\re\sqrt{z}|\leq \frac{1}{2}\re\sqrt{z}}(|k-\re\sqrt{z}|+1)^{-2}k^{2\nu(p')}\lesssim (\re\sqrt{z})^{2\nu(p')}.
\end{align}
For the second piece, we use that $|k-\re\sqrt{z}|+1>\frac{1}{2}|k-\re\sqrt{z}|+\frac{1}{4}\re\sqrt{z}$, leading to the better estimate
\begin{align}
    &\sum_{|k-\re\sqrt{z}|> \frac{1}{2}\re\sqrt{z}}(|k-\re\sqrt{z}|+1)^{-2}k^{2\nu(p')}\\
    &\lesssim \int_{-\infty}^{\infty}(|k-\re\sqrt{z}|+\re\sqrt{z})^{-2}k^{2\nu(p')}\rd k
    \lesssim (\re\sqrt{z})^{2\nu(p')-1}.
\end{align}
Altogether, we obtain \eqref{LpL2 2}.\\

For the proof of \eqref{LpL2 1}, we use that $\re\sqrt{\overline{z}}=\re\sqrt{z}$ and $\im\sqrt{\overline{z}}=-\im\sqrt{z}$. Moreover, for $\lambda_j\notin [\re\sqrt{z}-1/2,\re\sqrt{z}+1/2]$,
\begin{align}
 |\lambda_j^2-\overline{z}|=|\lambda_j+\sqrt{\overline{z}}||\lambda_j-\sqrt{\overline{z}}|\gtrsim (\re\sqrt{z})(|\lambda_j-\re\sqrt{z}|+1).
\end{align}
Thus, \eqref{LpL2 1} follows by repeating the proof of \eqref{LpL2 2}.
\end{proof}

We are now ready to prove Theorem \ref{theorem main}. For the proof, we will revert to the notation $\sigma(q)$, recalling that $\nu(p')=\sigma(q)$ when $1/q=1/p-1/p'$.

\begin{proof}[Proof of Theorem \ref{theorem main}]
 By the Birman--Schwinger principle, $z\in\C\setminus\spec(-\Delta_g)$ is an eigenvalue of $-\Delta_g+V$ if and only if $-1$ is an eigenvalue of the (compact) Birman--Schwinger operator $K(z):=|V|^{1/2}(-\Delta_g-z)^{-1}V^{1/2}$. In particular, if $z$ is an eigenvalue, then $K(z)$ has norm at least one, which implies that
\begin{align}\label{Birman--Schwinger principle}
    1\leq \|K(z)\|\leq \|V\|_{L^q(M)}\|(-\Delta_g-z)^{-1}\|_{L^{p}(M)\to L^{p'}(M)},
\end{align}
where $1/q=1/p-1/p'$. Thus, by \eqref{Birman--Schwinger principle} and the resolvent estimates \eqref{universal resolvent bound imsqrtz>1}, \eqref{universal resolvent bound imsqrtz<1}, if $z\in\spec(-\Delta_g+V)$, then either 
\begin{align}\label{either or}
 |z|^{\frac{1}{2}}(1+|z|)^{-\sigma(q)}\lesssim \|V\|_{L^q(M)}
\quad \mbox{or}\quad d(z)\lesssim (1+|z|)^{\sigma(q)}\|V\|_{L^q(M)}.
\end{align}
In the first case, \eqref{spectral inclusion} trivially follows.
The second case in \eqref{either or} occurs for $|\im\sqrt{z}|<1$, or equivalently, $z\in \Xi^c=\{z\in\C:(\im z)^2< 4(\re z+1)\}$, see \eqref{Xi}. The set $\Xi^c$ is contained in the union of a neighborhood of the origin and a parabolic region around $(1,\infty)$,
\begin{align}
 \Xi^c\subset \{z\in\C:|z|<\sqrt{8}\}\cup \{z\in\C:\re z>1,\,(\im z)^2< 8\re z\}.   
\end{align}
If $|z|<\sqrt{8}$, then the second inequality in \eqref{either or} implies that $z\in D(\lambda_0^2,Cr_0)$ (recalling that $\lambda_0=0$). If $\re z>1,\,(\im z)^2< 8\re z$, then there exists a unique $k\in\N$ such that $\re z\in [\lambda_{k-1}^2,\lambda_{k}^2)$. The second inequality in \eqref{either or} then implies that $d(z)\leq Cr_k$. Since $d(z)=\min\{|z-\lambda_{k-1}^2|,|z-\lambda_{k}^2|\}$ and since $\lim_{j\to\infty}\frac{\lambda_j}{\lambda_{j-1}}=1$ by Weyl's asymptotics, it follows that $z\in D(\lambda_{k-1}^2,r_{k-1})\cup D(\lambda_{k}^2,r_{k})$. This completes the proof of \eqref{spectral inclusion}, and hence of Theorem~\ref{theorem main} in the case $q>n/2$.\\

If $n\geq 3$ and $q=n/2$, then \eqref{universal resolvent bound imsqrtz>1} yields a uniform bound, i.e.
\begin{align}
    \|(-\Delta_g-z)^{-1}\|_{L^{\frac{2n}{n+2}}(M)\to L^{\frac{2n}{n-2}}(M)}
&\leq C_0,\quad |\im\sqrt{z}|\geq 1.
\end{align}
Therefore, \eqref{Birman--Schwinger principle} implies that 
$\|V\|_{L^{\frac{n}{2}}(M)}\geq C_0^{-1}$. Put differently, $z$ cannot be an eigenvalue if $|\im\sqrt{z}|\geq 1$ and $\|V\|_{L^{\frac{n}{2}}(M)}<C_0^{-1}$. Thus, if we set $c:=C_0^{-1}/2$, then all eigenvalues of $-\Delta_g+V$, $\|V\|_{L^{\frac{n}{2}}(M)}\leq c$, are contained in the region $|\im\sqrt{z}|<1$. Using \eqref{universal resolvent bound imsqrtz<1}, we obtain the refinement stated in \eqref{q=n/2 spectral inclusion}.
\end{proof}

\section{Optimality of eigenvalue bounds}
In this section, we prove Theorem \ref{theorem optimality sphere}. We will seek the potential in the form 
\begin{align}\label{wish list chi kappa}
    V=\kappa\chi^2,\quad 0\leq \chi=\chi_k\in L^{2q}(S^n),\quad \|\chi\|_{L^{2q}(S^n)}=1,\quad \kappa=\kappa_{k,\theta,\rho}\in\C, \quad 0\leq |\kappa|\lesssim \rho.
\end{align}

\subsection{Optimality on $S^n$}
We recall that
\begin{align}
    \spec(-\Delta_{S^n})=\{\lambda_k^2=k(k+n-1):k=0,1,2,\ldots\}
\end{align}
and that the eigenspace $H_k$ corresponding to the eigenvalue $k(k+n-1)$, the space of spherical harmonics of degree $k$, has dimension 
\begin{align}
    d_k:=\dim H_k=\begin{pmatrix}
        n+k\\n
    \end{pmatrix}
    -\begin{pmatrix}
        n+k-2\\n
    \end{pmatrix}\approx k^{n-1},
\end{align}
see e.g. \cite[Section 22.3]{MR1852334} or \cite[Section 3.4]{MR3186367}. \\

\textbf{Step 1:} By the Birman--Schwinger principle (see the first paragraph in the proof of Theorem \ref{theorem main}), it suffices to find $\chi,\kappa$ as above such that $-\kappa^{-1}$ is an eigenvalue of the compact operator $\chi(-\Delta_{S^n}-z_{k,\theta,\rho})^{-1}\chi$, where
\begin{align}\label{def z_k,theta,rho}
   z_{k,\theta,\rho}=k(k+n-1)+\rho\e^{\I\theta}(1+\mathcal{O}(\rho k^{2\sigma(q)-1})) k^{2\sigma(q)}.
\end{align}
In other words, we seek a zero $(z,\kappa)=(z_{k,\theta,\rho},\kappa_{k,\theta,\rho})$ of the operator-valued function 
\begin{align}\label{def A(zeta)}
    \mathcal{A}(z,\kappa):=I+\kappa\chi(-\Delta_{S^n}-z)^{-1}\chi,\quad z\in\C\setminus\spec(-\Delta_{S^n}),\quad|\kappa|\lesssim\rho.
\end{align}
For fixed $\kappa$, we denote 
\begin{align}
    \spec{\mathcal{A}(\cdot,\kappa)}:=\{z\in\C:\mathcal{A}(z,\kappa) \mbox{ is not invertible in } L^2(S^n)\}.
\end{align}
For $z\in D_k$ (defined in \eqref{D_k sphere}), we decompose $\chi(-\Delta_{S^n} - z)^{-1}\chi$ into a ‘resonant’ part and a ‘nonresonant’ part,
\begin{align}\label{def Ksing, Kreg}
K_{\rm r}(z):=\frac{\chi P_k\chi}{k(k+n-1)-z},\quad  K_{\rm nr}(z):=\chi(-\Delta_{S^n}-z)^{-1}\chi-K_{\rm r}(z), 
\end{align}
where $P_k$ is the orthogonal projection onto $H_k$.
Note that 
\begin{align}\label{for z in Dk}
    |\sqrt{z}-\sqrt{k(k+n-1)}|\lesssim \rho k^{2\sigma(q)-1}\quad\mbox{for all } z\in D_k.
\end{align}
Thus, if $q>n/2$ and \jc{$k\gg 1+\rho^{1/(1-2\sigma(q))}$}, or if $q=n/2$ and $\rho$ is sufficiently small, we have
\begin{align}
    |\sqrt{z}-\sqrt{k(k+n-1)}|<1/2,\quad \min_{j\neq k}|\sqrt{z}-\sqrt{j(j+n-1)}|\geq 1/2,
\end{align}
for all $z\in D_k$; the second inequality also uses that $|\sqrt{k(k+n-1)}-\sqrt{j(j+n-1)}|\gtrsim 1$ for $j\neq k$. Then \eqref{for z in Dk} implies that we may write
\begin{align}
 K_{\rm r}(z)&=\mathbf{1}(\sqrt{-\Delta_{S^n}}\in [\sqrt{k(k+n-1)}-1/2,\sqrt{k(k+n-1)}+1/2])(-\Delta_{S^n}-z)^{-1},\\   
 K_{\rm nr}(z)&=\mathbf{1}(\sqrt{-\Delta_{S^n}}\notin [\sqrt{k(k+n-1)}-1/2,\sqrt{k(k+n-1)}+1/2])(-\Delta_{S^n}-z)^{-1},   
\end{align}
and hence, by \eqref{frequencies close to re sqrt z}, \eqref{frequencies far from re sqrt z}, together with H\"older's inequality and the normalization condition on $\chi$ in \eqref{wish list chi kappa}, we have the bounds
\begin{align}
\sup_{z\in D_k}\|K_{\rm r}(z)\|_{L^{2}(S^n)\to L^{2}(S^n)}&\lesssim k^{2\sigma(q)}|k(k+n-1)-z|^{-1},\label{bound Ksing}\\
\sup_{z\in D_k}\|K_{\rm nr}(z)\|_{L^{2}(S^n)\to L^{2}(S^n)}&\lesssim k^{2\sigma(q)-1}.\label{bound Kreg}
\end{align}
We will choose $\chi$ so that \eqref{bound Ksing} becomes an approximate equality. \\

\textbf{Step 2:} Since \jc{$P_k$ has rank $d_k$ and $\chi$ is bounded, $\chi P_k\chi$ has rank at most $d_k$; in particular, it is a compact operator. Moreover, for any $f\in L^2(S^n)$, $\langle \chi P_k\chi f,f\rangle=\|P_k\chi f\|^2_{L^2(S^n)}\geq 0$, thus $\chi P_k\chi$ is self-adjoint and non-negative,
having at most $d_k$ nonzero eigenvalues $a_j(k)$} (counted with multiplicity),
\begin{align}\label{list of aj(k)}
    \|\chi P_k\chi\|=a_0(k)\geq a_1(k) \geq \ldots \geq a_{d_k}(k)\geq 0.
\end{align}
\jc{By the spectral theorem for compact, self-adjoint operators, $\chi P_k\chi$ has an orthonormal basis of eigenfunctions $\{f_{k,i}\}_{i=1}^{d_k}$. If $\{e_{k,i}\}_{i=1}^{d_k}$ is an orthonormal basis of $\Ran P_k$ (thus, $e_{k,i}$ are eigenfuncitons of $-\Delta_g$ corresponding to eigenvalue $\lambda_k^2$), then the functions $f_{k,j}$ lie in the linear span of $\{\chi e_{k,i}\}_{i=1}^{d_k}$. Moreover, the eigenvalues $a_j(k)$ can be computed by diagonalizing the Gram matrix $G=(G_{ij})$, with $G_{ij}=\langle\chi e_{k,i},\chi e_{k,j}\rangle$, $1\leq i,j\leq d_k$.}

\begin{lemma}\label{lemma chi P_k chi}
Let $n\geq 2$, $q>n/2$ or $n\geq 3$, $q\geq n/2$. For $k\in\N$ sufficiently large, there exists a nonnegative function~$\chi=\chi_k\in L^{2q}(S^n)$, $\|\chi\|_{L^{2q}(S^n)}=1$, and a constant $C_{a_0}>0$, such that 
 \begin{align}\label{chi saturates cluster bound sphere}
    C_{a_0}^{-1}k^{2\sigma(q)}\leq a_0(k)\leq C_{a_0}k^{2\sigma(q)}.
\end{align}
\end{lemma}

\begin{proof}
Sogge proved \cite[Theorem 4.1]{MR835795} that
\begin{align}\label{Sogge spherical harmonics}
    \|P_k\|_{L^p(S^n)\to L^2(S^n)}\approx k^{\nu(p')},
\end{align}
where we recall that $\nu(p')=\sigma(q)$ when $1/q=1/p-1/p'$. By a $TT^*$ argument, we also have
\begin{align}\label{Sogge spherical harmonics 2}
    \|P_k\|_{L^p(S^n)\to L^{p'}(S^n)}\approx k^{2\nu(p')}.
\end{align}
The upper bound in \eqref{Sogge spherical harmonics 2} and H\"older's inequality imply that for any nonnegative function $\chi$ with $\|\chi\|_{L^{2q}(S^n)}=1$,
\begin{align}\label{Holder/duality spherical harmonics}
\|\chi P_k\chi\|=\sup_{\|f\|_{L^{2}(S^n)}=1}\langle f,\chi P_k\chi f\rangle\leq \|P_k\|_{L^p(S^n)\to L^{p'}(S^n)}\sup_{\|f\|_{L^{2}(S^n)}=1}\|\chi f\|_{L^p(S^n)}\lesssim k^{2\sigma(q)}
\end{align}
since $1/p=1/(2q)+1/2$.\\

If we take the supremum of \eqref{Holder/duality spherical harmonics} over all normalized $\chi\in L^{2q}(S^n)$, we obtain
\begin{align}
 \sup_{\|\chi\|_{L^{2q}(S^n)}=1}\|\chi P_k\chi\|
 = \sup_{\substack{\|\chi\|_{L^{2q}(S^n)}=1\\ \|f\|_{L^{2}(S^n)}=1}}  
\langle \chi f,P_k\chi f\rangle
 \geq \sup_{\|g\|_{L^p(S^n)}=1}\langle g,P_k g\rangle
 =\|P_k\|_{L^p(S^n)\to L^{p'}(S^n)},
\end{align}
where the inequality follows from the fact that any normalized $g\in L^{p}(S^n)$ can be factored as $g=\chi f$, where $\chi$ and $f$ are normalized functions in $L^{2q}(S^n)$ and $L^{2}(S^n)$, respectively. Indeed, if we write $g=\e^{\I\arg g}|g|$, then we can take $\chi=|g|^{\frac{p}{2q}}$ and $f=\e^{\I\arg g}|g|^{\frac{p}{2}}$. It follows from \eqref{Sogge spherical harmonics 2} that there exists a normalized function $\chi\in L^{2q}(S^n)$ such that we have $\|\chi P_k\chi\|\gtrsim k^{2\sigma(q)}$. The lemma is proved since the supremum can be taken over nonnegative functions $\chi$.
\end{proof}

\begin{remark}
(i) The function $\chi$ can be constructed as follows. 
By \eqref{Sogge spherical harmonics}, there exists $\psi_k\in L^p(S^n)$ such that
\begin{align}\label{Sogge spherical harmonics 3}
    \frac{\|P_k\psi_k\|_{L^{2}(S^n)}}{\|\psi_k\|_{L^p(S^n)}}\approx k^{\sigma(q)},
\end{align}
and we normalize $\psi_k$ so that $\|\psi_k\|_{L^p(S^n)}=1$. 
With this choice, we have
\begin{align}
\langle\psi_k,P_k\psi_k\rangle=\|P_k\psi_k\|_{L^2(S^n)}^2\approx k^{2\sigma(q)}.
\end{align}
Similarly as in the proof of Lemma \ref{lemma chi P_k chi}, setting $\chi_k=|\psi_k|^{\frac{p}{2q}}$ and $f_k=\e^{\I\arg \psi_k}|\psi_k|^{\frac{p}{2}}$, we then have $\psi_k=\chi_k f_k$, $\|\chi_k\|_{L^{2q}(S^n)}=1$, $\|f_k\|_{L^{2}(S^n)}=1$, and
\begin{align}
 \|\chi_k P_k\chi_k\|\geq \langle \chi_k f_k,P_k\chi_k f_k\rangle\approx k^{2\sigma(q)}.
\end{align}
We refer to \cite[Chapter 5.1]{MR3645429} for the construction of $\psi_k$ in \eqref{Sogge spherical harmonics 3}, which depends on whether $2(n+1)/(n-1)\leq p'\leq \infty$ or $2\leq p'\leq 2(n+1)/(n-1)$. This construction works on any compact manifold and shows that Sogge's unit-width spectral cluster bounds \eqref{Sogge 1}, \eqref{Sogge 2} are always optimal. 

\noindent (ii)
On $S^n$, the construction is even more explicit, at least when $2\leq p'\leq 2(n+1)/(n-1)$.
Sogge \cite{MR835795} showed that in this case, \eqref{Sogge spherical harmonics 3} is saturated by the highest weight spherical harmonic $Q_k$ of degree $k$. Then $\chi_k=|Q_k|^{\frac{p}{2q}}$ is completely explicit. Sogge also showed that the dual of~\eqref{Sogge spherical harmonics} is saturated by the $k$-th zonal spherical harmonic $Z_k$ for $2(n+1)/(n-1)\leq p'\leq \infty$, but $Z_k$ cannot be directly used in \eqref{Sogge spherical harmonics 3} \jc{since it does not have the stated $L^2$ to $L^p$ norm ratio in that range of $p$.}
\end{remark}

\textbf{Step 3:}
We will compare the zeros of $\mathcal{A}(\cdot,\kappa)$ (defined in \eqref{def A(zeta)}) to the zeros $z_{{\rm r},j}(\kappa)$ of 
the following operator-valued function,
\begin{align}\label{def. Asing}
    \mathcal{A}_{\rm r}(z,\kappa):=I+\kappa K_{\rm r}(z),\quad z\in\C\setminus\spec(-\Delta_{S^n}),\quad|\kappa|\lesssim\rho,
\end{align}
where $K_{\rm r}$ was defined in \eqref{def Ksing, Kreg}.
These zeros are given by 
\begin{align}\label{z_rj}
    \spec(\mathcal{A}_{\rm r}(z,\kappa))=\{z_{{\rm r},j}(\kappa):=k(k+n-1)+\kappa a_{j}(k): j=0,\ldots,d_k\}.
\end{align}
Setting 
\begin{align}\label{def. kappa_0}
\kappa_{0}:=\rho \e^{\I\theta} k^{2\sigma(q)}a_0(k)^{-1}\in D(0,C_{a_0}\rho),
\end{align}
we find that 
\begin{align}\label{z_r0(kappa_r)}
    z_{{\rm r},0}(\kappa_{0})=k(k+n-1)+\rho\e^{\I\theta}k^{2\sigma(q)},\quad |z_{{\rm r},0}(\kappa_{0})-z_{k,\theta,\rho}|\lesssim \rho^2 k^{4\sigma(q)-1},
\end{align}
where $z_{k,\theta,\rho}$ is given by \eqref{def z_k,theta,rho}.
Hence, by the triangle inequality, the proof of the first part of Theorem \ref{theorem optimality sphere} would be complete if we could show that
\begin{align}\label{zero z of A near zthetarho}
    D(z_{{\rm r},0}(\kappa_{0}),C\rho^2 k^{4\sigma(q)-1})\cap\spec{\mathcal{A}(\cdot,\kappa_{0})}\neq \emptyset
\end{align}
for sufficiently large $k$ and some universal constant $C$.

To this end, we will use an operator version of Rouché’s theorem, due to Gohberg and Sigal \cite[Theorem 2.2]{MR313856}. It implies that if $\gamma$ is a simple closed rectifiable contour bounding a domain $U\subset \C$ and if
\begin{align}\label{smallness condition in Rouche}
\max_{w\in\gamma}\|\mathcal{A}_{\rm r}(w,\kappa_{0})^{-1}(\mathcal{A}(w,\kappa_{0})-\mathcal{A}_{\rm r}(w,\kappa_{0}))\|<1,
    \end{align}
then
    \begin{align}\label{traces equal}
        \frac{1}{2\pi\I}\Tr\oint_{\gamma}\partial_w \mathcal{A}_{\rm r}(w,\kappa_{0}) \mathcal{A}_{\rm r}(w,\kappa_{0})^{-1}\rd w
        =\frac{1}{2\pi\I}\Tr\oint_{\gamma}\partial_w \mathcal{A}(w,\kappa_{0}) \mathcal{A}(w,\kappa_{0})^{-1}\rd w.
    \end{align}
Ideally, we would like to choose the contour $\gamma=\partial U$ to be a circle with center $z_{{\rm r},0}(\kappa_{0})$ and radius $C\rho^2 k^{4\sigma(q)-1}$. However, due to their $k$-dependence, it is not clear how to avoid zeros of $\mathcal{A}_{\rm r}$ and $\mathcal{A}$. The main problem is that we have no uniform (independent of $k$) lower bound on the gap between the largest eigenvalue $a_0(k)$ of $\chi P_k\chi$ and the next largest distinct eigenvalue (recall that, in \eqref{list of aj(k)}, the eigenvalues $a_j(k)$ were counted with multiplicity.)

\begin{lemma}\label{lemma curve}
Let $n\geq 2$, $q>n/2$ or $n\geq 3$, $q\geq n/2$, let $\kappa_0$, $z_{{\rm r},0}(\kappa_{0})$ be as as in \eqref{def. kappa_0}, \eqref{z_r0(kappa_r)}, and assume that $\|\chi\|_{L^{2q}(S^n)}=1$. 
For $\eps,\rho>0$, $\eps$ sufficiently small, and $k\in \N$ sufficiently large, such that $\rho\eps^{-1}k^{2\sigma(q)-1}\ll 1$, there exists a circle $\gamma=\partial U$ with center $z_{{\rm r},0}(\kappa_{0})$ and radius $\mathcal{O}(\eps\rho)k^{2\sigma(q)}$, which does not intersect any zeros of $\mathcal{A}_{\rm r}(\cdot,\kappa_{0})$ or $\mathcal{A}(\cdot,\kappa_{0})$, and such that 
\begin{align}\label{number of eigenvalues of A=number of eigenvalues of A_r}
    \#(\spec{\mathcal{A}(\cdot,\kappa_0)}\cap U)=\#(\spec{\mathcal{A}_{\rm r}(\cdot,\kappa_0)}\cap U). 
\end{align}
\end{lemma}

\begin{proof}[Conclusion of the proof of Theorem \ref{theorem optimality sphere} i)]
Since $z_{{\rm r},0}(\kappa_{0})\in \spec{\mathcal{A}_{\rm r}(\cdot,\kappa_0)}$, the right-hand side of \eqref{number of eigenvalues of A=number of eigenvalues of A_r} is $\geq 1$.
Since we may take $\eps=C\rho k^{2\sigma(q)-1}$ for some $C\gg 1$ in Lemma \ref{lemma curve}, it follows that
\begin{align}
    \#(\spec{\mathcal{A}(\cdot,\kappa_0)}\cap D(z_{{\rm r},0}(\kappa_{0}),C\rho^2 k^{4\sigma(q)-1}))\geq 1,
\end{align}
 which means that \eqref{zero z of A near zthetarho} holds.
 \end{proof}

\begin{proof}[Proof of Lemma \ref{lemma curve}]
1. We will use the following bound on the eigenvalues $a_j(k)$ of $\chi P_k \chi$,
\begin{align}\label{Frank Sabin orthocluster}
 \Big(\sum_j a_j(k)^{\beta(q)}\Big)^{1/\beta(q)}\leq C_{\beta(q)} k^{2\sigma(q)},   
\end{align}
for some $\beta(q)\in [1,\infty)$, which is a consequence of the Schatten norm bounds of Frank and Sabin~\cite{MR3682666}. Specifically, this follows from \cite[Remark 6]{MR3682666} in the case $M=S^n$ and $\lambda\approx k$. The Schatten exponent $\beta(q)$ is given explicitly by 
\begin{align}\label{def. beta(q)}
    \beta(q)=\begin{cases}
        \frac{(n-1)q}{n-q}\quad &\mbox{if } \frac{n}{2}\leq  q\leq \frac{n+1}{2},\\
        2q\quad &\mbox{if } \frac{n+1}{2}\leq q\leq \infty, 
    \end{cases} 
\end{align}
however, the exact value of $\beta(q)$ is not important to us here. From \eqref{Frank Sabin orthocluster} and the monotonicity of $a_j(k)$, we get
\begin{align}
 a_j(k)\leq C_{\beta(q)} j^{-1/\beta(q)} k^{2\sigma(q)}\leq C_{\beta(q)}C_{a_0} j^{-1/\beta(q)} a_0(k),
\end{align}
where the second estimate follows from Lemma \ref{lemma chi P_k chi}.
Let $j_0=j_0(q)$ be the first integer such that $j_0>(C_{\beta(q)}C_{a_0})^{\beta(q)}$. 
Then 
\begin{align}
    a_0(k)-a_{j_0}(k)\geq C_1a_0(k),\quad \mbox{where}\quad C_1:=(1-C_{\beta(q)}C_{a_0}j_0^{-1/\beta(q)})>0.
\end{align}
Moreover, there are at most $j_0$ eigenvalues $a_1(k),\ldots,a_{j_0}(k)$ in the interval $[a_{j_0}(k),a_0(k)]$. 
If we rescale $\tilde{a}_j(k):=a_j(k)/a_0(k)$, then
\begin{align}
0\leq \tilde{a}_{j_0}(k)\leq \ldots \leq \tilde{a}_{1}(k)\leq  \tilde{a}_{0}(k)=1,\quad   \tilde{a}_{j_0}(k)\leq 1-C_1.
\end{align}
Let $0<\eps<C_1/2$, so that $I_{\eps}:=[1-2\eps,1]$ is a subinterval of $[\tilde{a}_{j_0}(k),1]$ and thus contains at most $j_0$ distinct points of the sequence $\{\tilde{a}_j(k)\}_{j=0}^{d_k}$. Let
\begin{align}
    I_{\eps,m}:=[1-\frac{2\eps m}{j_0+1},1-\frac{2\eps (m-1)}{j_0+1}],\quad m=1,2,\ldots, j_0+1.
\end{align}
By the pigeonhole principle, there exists at least one interval $I_{\eps,m(k)}$ that contains no point of $\{\tilde{a}_j(k)\}_{j=0}^{d_k}$. Let $\tilde{x}(k)$ be the midpoint of this interval. By construction,
\begin{align}\label{xtilde}
    \min_{0\leq j\leq d_k}|\tilde{x}(k)-\tilde{a}_j(k)|\geq \frac{\eps}{j_0+1},\quad 2\eps\geq 1-\tilde{x}(k)\geq \frac{\eps}{j_0+1}.
\end{align}
We now define 
\begin{align}\label{def. U(k,eps)}
    U:=D(z_{{\rm r},0}(\kappa_{0}),|\kappa_{0}|(1-\tilde{x}(k))a_0(k))\subset D_k,\quad \gamma:=\partial U.
\end{align}
By \eqref{xtilde}, \eqref{def. kappa_0}, we can write
\begin{align}\label{def. U(k,eps) 2}
    U=D(z_{{\rm r},0}(\kappa_{0}),\mathcal{O}(\eps\rho) k^{2\sigma(q)}).
\end{align}

2. To prove $\gamma\cap \spec(\mathcal{A}_{\rm r}(\cdot,\kappa_0))=\emptyset$, we first recall that
\begin{align}
 \spec(\mathcal{A}_{\rm r}(\cdot,\kappa_{0}))=\{z_{{\rm r},j}(\kappa_{0})\}_{j=0}^{d_k},   
\end{align}
where $z_{{\rm r},j}(\kappa)$ are given by \eqref{z_rj}. We fix $j\in \{0,\ldots,d_k\}$ and $w\in\gamma$, and write
\begin{align}\label{w in gamma}
    w=z_{{\rm r},0}(\kappa_{0})+\e^{\I\varphi}\kappa_{0}(1-\tilde{x}(k))a_0(k))
\end{align}
for some $\varphi\in [0,2\pi)$. Then
\begin{align}\label{w-zrj}
    |w-z_{{\rm r},j}(\kappa_{0})|&\geq |w-z_{{\rm r},0}(\kappa_{0})|-|z_{{\rm r},j}(\kappa_{0})-z_{{\rm r},0}(\kappa_{0})|\\
    &\geq |\kappa_{0}|(1-\tilde{x}(k))a_0(k)-|\kappa_{0}|(a_0(k)-a_j(k))\\
    &=|\kappa_{0}|(\tilde{a}_j(k)-\tilde{x}(k))a_0(k)
    \geq \frac{\eps \rho k^{2\sigma(q)}}{j_0+1},
\end{align}
where we used the first bound of \eqref{xtilde} and the definition of $\kappa_0$, see \eqref{def. kappa_0}, in the last inequality. This implies that
\begin{align}
    \dist(\gamma,\spec(\mathcal{A}_{\rm r}(\cdot,\kappa_{0})))&\geq \frac{\eps \rho k^{2\sigma(q)}}{j_0+1}.\label{distance to curve bounded below}
\end{align}
and thus $\gamma\cap \spec(\mathcal{A}_{\rm r}(\cdot,\kappa_0))=\emptyset$.\\

3. To prove $\gamma\cap \spec(\mathcal{A}(\cdot,\kappa_0))=\emptyset$, let $w\in\gamma$, as in \eqref{w in gamma}. Since $K_{\rm r}(w)$ is a normal operator, we have
\begin{align}\label{A_r inverse norm}
    \|\mathcal{A}_{\rm r}(w,\kappa_0)^{-1}\|&=\max_{0\leq j\leq d_k}|\Big(1+\frac{\kappa_{0}a_j(k)}{k(k+n-1)-w}\Big)^{-1}|\\
    &=|k(k+n-1)-w|\max_{0\leq j\leq d_k}|w-z_{{\rm r},j}(\kappa_0)|^{-1}.
\end{align}
By \eqref{z_rj}, \eqref{def. U(k,eps) 2} and the 
second formula in \eqref{xtilde}, we have
\begin{align}\label{k(k+n-1)-w}
|k(k+n-1)-w|&\leq |k(k+n-1)-z_{{\rm r},0}(\kappa_{0})|+|w-z_{{\rm r},0}(\kappa_{0})|\\
&\leq  |\kappa_0| a_0(k)+|\kappa_{0}|a_0(k)(1-\tilde{x}(k))\\
&\leq \rho k^{2\sigma(q)}(1+2\eps).
\end{align}
Together with \eqref{w-zrj}, \eqref{A_r inverse norm}, we then obtain
\begin{align}\label{A_r inverse norm 2}
 \|\mathcal{A}_{\rm r}(w,\kappa_0)^{-1}\|\leq (1+2\eps)(j_0+1)\eps^{-1}.    
\end{align}
Then, since
\begin{align}
    \mathcal{A}(w,\kappa_0)=\mathcal{A}_{\rm r}(w,\kappa_0)(I+\mathcal{A}_{\rm r}(w,\kappa_0)^{-1}(\mathcal{A}(w,\kappa_0))-\mathcal{A}_{\rm r}(w,\kappa_0)),
\end{align}
and, by \eqref{bound Kreg}, 
\begin{align}
    \|\mathcal{A}(w,\kappa_0)-\mathcal{A}_{\rm r}(w,\kappa_0)\|=|\kappa_0|\|K_{\rm nr}\|
    \lesssim \rho k^{2\sigma(q)-1},
\end{align}
it follows that $\mathcal{A}(w,\kappa_0)$ is invertible, and 
\begin{align}
 \|\mathcal{A}(w,\kappa_0)^{-1}\|\leq (1+2\eps)(j_0+1)\eps^{-1}(1-\mathcal{O}(\rho\eps^{-1}k^{2\sigma(q)-1}))^{-1}.
\end{align}
This proves that $\gamma\cap \spec(\mathcal{A}(\cdot,\kappa_0))=\emptyset$.\\

4. Finally, we check that the smallness condition \eqref{smallness condition in Rouche} holds. 
By \eqref{A_r inverse norm 2} and \eqref{bound Kreg},
we have
\begin{align}
&\max_{w\in\gamma}\|\mathcal{A}_{\rm r}(w,\kappa_{0})^{-1}(\mathcal{A}(w,\kappa_{0})-\mathcal{A}_{\rm r}(w,\kappa_{0}))\|\\
&\leq |\kappa_0|\|\mathcal{A}_{\rm r}(w,\kappa_{0})^{-1}\| \max_{w\in\gamma}\|K_{\rm nr}\|\\
&\lesssim\rho\eps^{-1}k^{2\sigma(q)-1}\ll 1.
\end{align}
Now, \cite[Theorem 2.2]{MR313856} yields \eqref{traces equal}, which is equivalent to \eqref{number of eigenvalues of A=number of eigenvalues of A_r}.
\end{proof}

\subsection{Optimality on Zoll manifolds}
The proof of the second part of Theorem \ref{theorem optimality sphere} is analogous to that of the first part, with a few modifications, which we shall point out below. As mentioned in Remark \ref{remark theorem optimality}, optimality for $q=\infty$ is trivial, therefore we exclude this case here.\\

\textbf{Step 1:} We replace $\Delta_{S^n}$ with $\Delta_g$ in the definition \eqref{def A(zeta)} and $z_{k,\theta,\rho}$ in \eqref{def z_k,theta,rho} by
\begin{align}
   z_{k,\theta,\rho}=(k+\alpha)^2+\rho\e^{\I\theta}(1+\mathcal{O}(\rho k^{2\sigma(q)-1})) k^{2\sigma(q)}.
\end{align}
Note that, due to the additional assumption $k^{2\sigma(q)}\rho\gg 1$ in the Zoll case, the eigenvalues $\lambda_{k,j}^2$ of $-\Delta_g$ corresponding to the $k$-th cluster are contained inside the disks $\tilde{D}_k$ defined in~\eqref{D_k tilde Zoll}, and these disks are mutually disjoint. 
The situation is similar to that on the sphere, the only difference being that a single eigenvalue is now replaced by a whole cluster of eigenvalues. For this reason, we use a different splitting $K(z)=K_{\rm r}(z)+K_{\rm nr}(z)$, namely
\begin{align}\label{def K_r Zoll}
 K_{\rm r}(z)&:=\frac{\chi\mathbf{1}(\sqrt{-\Delta_g}\in[k+\alpha-A/k,k+\alpha+  A/k])\chi}{(k+\alpha)^2-z}\\
 &=\frac{\chi\mathbf{1}(\sqrt{-\Delta_g}\in[k+\alpha-1/2,k+\alpha+1/2])\chi}{(k+\alpha)^2-z},
 \end{align}
 where the second equality holds due to the clustering of the eigenvalues.

 We first establish an analogue of \eqref{bound Kreg}. Note that we only used this in the last step of the proof of Lemma \ref{lemma curve} over the curve $\gamma$, not over the whole disk $D_k$.

\begin{lemma}\label{lemma Zoll}
Assume that $n/2<q<\infty$ and $k\gg 1$ or $q=n/2$, $n\geq 3$, and $\rho\ll 1$. Then
    \begin{align}\label{bound Knr Zoll}
    \|K_{\rm nr}(z)\|&\lesssim |(k+\alpha)^2-z|^{-2}k^{2\sigma(q)}+k^{2\sigma(q)-1},
\end{align}
uniformly over all $z\in \tilde{D}_k$ with $|z-(k+\alpha)^2|\gg_A 1$ and $\|\chi\|_{L^{2q}(M)}\lesssim 1$.
\end{lemma} 

\begin{proof}
 We further split $K_{\rm nr}(z)=K^{(1)}_{\rm nr}(z)+K^{(2)}_{\rm nr}(z)$, with
 \begin{align}
K^{(1)}_{\rm nr}(z)&=\sum_{j=1}^{d_k}\left(\frac{1}{\lambda_{k,j}^2-z}-\frac{1}{(k+\alpha)^2-z}\right)\chi E_{k,j}\chi,\\
 K^{(2)}_{\rm nr}(z)&=\chi\mathbf{1}(\sqrt{-\Delta_g}\notin[k+\alpha-1/2,k+\alpha+1/2])(-\Delta_g-z)^{-1}\chi,   
\end{align}
where $E_{k,j}$ is the orthogonal projection onto the eigenfunction of $-\Delta_g$ corresponding to $\lambda_{k,j}^2$.
As a consequence of \eqref{frequencies far from re sqrt z}, we have an analogous bound to \eqref{bound Kreg},
\begin{align}\label{bound K^2_nr}
    \sup_{z\in \tilde{D}_k}\|K^{(2)}_{\rm nr}(z)\|&\lesssim k^{2\sigma(q)-1}.
\end{align}
If we set
\begin{align}
   m_k(\tau,z):=\left(\frac{1}{\tau^2-z}-\frac{1}{(k+\alpha)^2-z}\right)\mathbf{1}(\tau\in [k+\alpha-A/k,k+\alpha+  A/k]),
   \end{align}
then we have
\begin{align}
 \|K^{(1)}_{\rm nr}(z)\|&=\|\chi m_k(\sqrt{-\Delta
_g},z)\chi\|\\
&\leq \|\chi\|^2_{L^{2q}(M)}\|m_k(\sqrt{-\Delta
_g},z)\|_{L^p(M)\to L^{p'}(M)}\\
&\lesssim \||m_k(\sqrt{-\Delta
_g},z)|^{1/2}\|_{L^p(M)\to L^{2}(M)}^2.
\end{align}
By $L^2$ orthogonality, for $z\in \tilde{D}_k$, $|z-(k+\alpha)^2|\gg A$,
\begin{align}
&\||m_k(\sqrt{-\Delta_g},z)|^{1/2}f\|_{L^{2}(M)}^2 
=\sum_{j=1}^{d_k}\left|\frac{1}{\lambda_{k,j}^2-z}-\frac{1}{(k+\alpha)^2-z}\right| \|E_{k,j}f\|_{L^{2}(M)}^2\\
&\leq \max_{1\leq j\leq d_k}\left|\frac{1}{\lambda_{k,j}^2-z}-\frac{1}{(k+\alpha)^2-z}\right| \|\mathbf{1}(\sqrt{-\Delta_g}\in[k+\alpha-A/k,k+\alpha+  A/k])f\|_{L^{2}(M)}^2\\
&\lesssim |(k+\alpha)^2-z|^{-2}k^{2\sigma(q)}\|f\|_{L^{p}(M)}^2,
\end{align}
which implies
\begin{align}\label{bound K^1_nr}
   \|K^{(1)}_{\rm nr}(z)\|\lesssim |(k+\alpha)^2-z|^{-2}k^{2\sigma(q)}.
\end{align}   
The proof is complete.
\end{proof}

\textbf{Step 2:} Instead of $\chi P_k\chi$ in the case of $S^n$, we now consider the nominator in the second line of \eqref{def K_r Zoll} and denote its eigenvalues by $a_j(k)$, viz.
\begin{align}\label{list of aj(k) Zoll}
    \|\chi\mathbf{1}(\sqrt{-\Delta_g}\in[k+\alpha-1/2,k+\alpha+1/2])\chi\|=a_0(k)\geq a_1(k) \geq \ldots \geq a_{d_k}(k)\geq 0.
\end{align}
We now show that Lemma \ref{lemma chi P_k chi} still holds in this case. Let us set $\kappa=k+\alpha$. The proof of optimality of Sogge’s spectral cluster bounds shows that
\begin{align}
    \|\mathbf{1}(\sqrt{-\Delta_g}\in[\kappa-1/2,\kappa+1/2])\|_{L^p(M)\to L^2(M)}\geq c \kappa^{-n+n/p}\left(N(\kappa+1/2)-N(\kappa-1/2)\right)^{1/2}
\end{align}
for some positive constant $c>0$ if $1\leq p\leq 2$ (see the display immediately after (5.1.12') in~\cite{MR3645429}). By the sharp Weyl law and the clustering property, we have 
\begin{align}
  c_1:=\liminf_{\kappa\to\infty}  \kappa^{-n+1}\left(N(\kappa+1/2)-N(\kappa-1/2)\right)>0.
\end{align}
This implies that, for $\kappa\gg 1$,  
\begin{align}
 \left(N(\kappa+1/2)-N(\kappa-1/2)\right)\geq \frac{c_1}{2}\kappa^{n-1}
\end{align}
and thus
\begin{align}\label{LpL2 Zoll 1}
    \|\mathbf{1}(\sqrt{-\Delta_g}\in[\kappa-1/2,\kappa+1/2])\|_{L^p(M)\to L^{2}(M)}\geq \frac{cc_1}{\sqrt{2}} \kappa^{-n+n/p+\frac{n-1}{2}}
\end{align}
for $\kappa\gg 1$, or equivalently, $k\gg 1$. Since $1/q=1/p-1/p'$, the exponent of $\kappa$ equals $n/(2q)-1/2$, which is $\sigma(q)$ for $q\leq (n+1)/2$. In the case $q\geq (n+1)/2$, we appeal to \cite[(5.1.13)]{MR3645429}, which for $\kappa\gg 1$ gives the lower bound
\begin{align}\label{LpL2 Zoll 2}
    \|\mathbf{1}(\sqrt{-\Delta_g}\in[\kappa-1/2,\kappa+1/2])\|_{L^p(M)\to L^2(M)}\geq c \kappa^{(n-1)(\frac{1}{2p}-\frac{1}{4})},\quad 1\leq p\leq 2.
\end{align}
The exponent of $\kappa$ equals $(n-1)/(4q)$, which is $\sigma(q)$ for $q\geq (n+1)/2$. By a $TT^*$ argument, the combination of \eqref{LpL2 Zoll 1}, \eqref{LpL2 Zoll 2} yields
\begin{align}\label{LpL2 Zoll 3}
    \|\mathbf{1}(\sqrt{-\Delta_g}\in[k+\alpha-1/2,k+\alpha+1/2])\|_{L^p(M)\to L^{p'}(M)}\gtrsim k^{2\sigma(q)},\quad 1\leq q\leq\infty,
\end{align}
for $k\gg 1$. This is the analogue, in the case of general Zoll manifolds, of \eqref{Sogge spherical harmonics 2} on $S^n$. Repeating the proof of Lemma \ref{lemma chi P_k chi}, we find 
a nonnegative function $\chi=\chi_k\in L^{2q}(M)$, $\|\chi\|_{L^{2q}(M)}=1$ such that 
 \begin{align}\label{chi saturates cluster bound Zoll}
   a_0(k)=\|\chi\mathbf{1}(\sqrt{-\Delta_g}\in[k+\alpha-1/2,k+\alpha+1/2])\chi\|_{L^2(M)\to L^{2}(M)}\approx k^{2\sigma(q)}.
\end{align}

\textbf{Step 3:}
Recalling that $a_j(k)$ now denote the eigenvalues of the operator in \eqref{chi saturates cluster bound Zoll}, it follows that
\begin{align}
    \spec(\mathcal{A}_{\rm r}(z,\kappa))=\{z_{{\rm r},j}(\kappa):=(k+\alpha)^2+\kappa a_{j}(k): j=0,\ldots,d_k\},\quad a_0(k)\approx k^{2\sigma(q)},
\end{align}
where $\mathcal{A}_{\rm r}(z,\kappa)$ is defined analogously to \eqref{def. Asing}, but now with $K_{\rm r}(z)$ from \eqref{def K_r Zoll}. Choosing $\kappa_0$ as in \eqref{def. kappa_0}, we get
\begin{align}\label{z_r0 Zoll}
    z_{{\rm r},0}(\kappa_{0})=(k+\alpha)^2+\rho\e^{\I\theta}k^{2\sigma(q)}.
\end{align}
Since the Schatten norm bound~\eqref{Frank Sabin orthocluster} of Frank and Sabin~\cite{MR3682666} for unit size spectral clusters holds on any manifold, 
we can repeat Steps 1-3 in the proof of Lemma \ref{lemma curve}. This gives us a curve
\begin{align}
   \gamma=\partial D(z_{{\rm r},0}(\kappa_{0}),\mathcal{O}(\eps\rho) k^{2\sigma(q)})
\end{align}
with $z_{{\rm r},0}(\kappa_{0})$ as in \eqref{z_r0 Zoll} and such that
\begin{align}
    \gamma\cap \spec(\mathcal{A}_{\rm r}(\cdot,\kappa_0))=\emptyset,\quad \gamma\cap \spec(\mathcal{A}(\cdot,\kappa_0))=\emptyset,
\end{align}
where $\mathcal{A}(z,\kappa_0)$ is defined analogously to \eqref{def A(zeta)}. 

To verify the smallness condition \eqref{smallness condition in Rouche}, we use Lemma \ref{lemma Zoll}, which is applicable since, for $w\in\gamma$, 
\begin{align}
    |w-(k+\alpha)^2|\geq |z_{{\rm r},0}(\kappa_{0})-(k+\alpha)^2|-|z_{{\rm r},0}(\kappa_{0})-w|\gtrsim \rho(1-\mathcal{O}(\eps))k^{2\sigma(q)}\gg_A 1
\end{align}
if $k\gg 1$ (we recall that $\sigma(q)>0$ for $q<\infty$). Since $K_{\rm r}(w)$ is a normal operator, the analogue of \eqref{A_r inverse norm}, with $(k+\alpha)^2$ in place of $k(k+n-1)$, holds, and thus (c.f. \eqref{A_r inverse norm 2})
\begin{align*}
     \|\mathcal{A}_{\rm r}(w,\kappa_0)^{-1}\|\lesssim \eps^{-1}.
\end{align*}
This, together with \eqref{bound Knr Zoll} and the bound $|\kappa_0|\lesssim\rho$ (c.f. \eqref{def. kappa_0}), yields
\begin{align}
&\max_{w\in\gamma}\|A_{\rm r}(w,\kappa_{0})^{-1}(A(w,\kappa_{0})-A_{\rm nr}(w,\kappa_{0}))\|\\
&\leq |\kappa_0|\|A_{\rm r}(w,\kappa_{0})^{-1}\| \max_{w\in\gamma}\|K_{\rm nr}\|\\
&\lesssim\rho\eps^{-1}(\max_{w\in\gamma}|(k+\alpha)^2-w|^{-2}k^{2\sigma(q)}+k^{2\sigma(q)-1})\\
&\lesssim \rho^{-1}\eps^{-1}k^{-2\sigma(q)}+\rho\eps^{-1}k^{2\sigma(q)-1}
\ll 1,
\end{align}
where the last inequality holds for $k\gg 1$ if $q>n/2$ and for $\rho\ll 1$, $k^{2\sigma(q)}\rho\gg 1$ for $q=n/2$, $n\geq 3$. The latter condition is satisfied by assumption.

 Thus, Lemma \ref{lemma curve} is also proved for the Zoll case, and the second part of Theorem~\ref{theorem optimality sphere} follows in the same way as the first.

\section{Scaling considerations}\label{Section Scaling considerations}
Consider $M=\mathbb{T}^n$, the $n$-dimensional flat torus, viewed as $(-1/2,1/2]^n\subset\R^n$. For $L>0$, $p\in [1,\infty)$, the space $L^p(L\mathbb{T}^n)$ consists of $p$-integrable functions on $Q_L:=(-L/2,L/2]^n$ satisfying periodic boundary conditions. Extension by zero yields an embedding 
\begin{align}\label{torus embedding}
    L^p(L\mathbb{T}^n)\subset L^p(\R^n).
 \end{align}   
Conversely, multiplication with the characteristic function of $Q_L$ defines a projection $$\mathbf{1}_{Q_L}:L^p(\R^n)\to L^p(L\mathbb{T}^n).$$
We define the unitary transformation
\begin{align}
    U_L:L^2(\mathbb{T}^n)\to L^2(L\mathbb{T}^n),\quad (U_Lf)(x):=L^{-\frac{n}{2}}f(x/L).
\end{align}
Then
\begin{align}
    L^{-2}U_L(-\Delta+V)U_L^{-1}=
    -\Delta+L^{-2}V(\cdot/L),
\end{align}
which implies \eqref{spectra scaled}. We will consider $L^2V(L\cdot)$ in place of $V$, in which case \eqref{spectra scaled} becomes
\begin{align}\label{spectra scaled 2}
    \spec_{L^2(L\mathbb{T}^n)}(-\Delta+V)=L^{-2}\spec_{L^2(\mathbb{T}^n)}(-\Delta+L^2V(L\cdot)).
\end{align} 
We assume $V\in L^q(LM)$ for some $q\in(n/2,\infty)$; the case $q=n/2$ can be handled similarly (under an additional smallness condition), but we will not consider it here. Theorem \ref{theorem main}, and more specifically \eqref{spectral inclusion}, implies that the spectrum of $-\Delta+V$ is contained in the set
\begin{align}\label{spectral inclusion scaled}
\bigcup_{k=0}^{\infty}D(L^{-2}\lambda_k^2,CL^{-2}r_k)\cup \{z\in\C: L|z|^{\frac{1}{2}}(1+L^2|z|)^{-\sigma(q)}\leq C L^{2-\frac{n}{q}}\|V\|_{L^q(L\mathbb{T}^n)}\},
\end{align}
where we used the identity $L^2\|V(L\cdot)\|_{L^q(\mathbb{T}^n)}=L^{2-\frac{n}{q}}\|V\|_{L^q(L\mathbb{T}^n)}$, and where 
\begin{align}
    r_0=L^{2-\frac{n}{q}}\|V\|_{L^q(L\mathbb{T}^n)},\quad r_k=\lambda_k^{2\sigma(q)}L^{2-\frac{n}{q}}\|V\|_{L^q(L\mathbb{T}^n)},\quad k\geq 1.
\end{align}

\begin{lemma}\label{lemma scaling limit}
 Let $M =\mathbb{T}^n$ be the flat torus. Assume that $(z_L)_{L\in\N}\subset\C$ is a convergent sequence (as $L\to\infty$).\\

 \noindent 1. If $q\in (n/2,\infty)$ and $z_L\in \bigcup_{k=0}^{\infty}D(L^{-2}\lambda_k^2,CL^{-2}r_k)$, then 
 $\lim_{L\to\infty}z_L\in [0,\infty)$.\\

 \noindent 2. If $q\in (n/2,(n+1)/2)$ and 
\begin{align}\label{second set expressed as inequality for zL}
    L|z_L|^{\frac{1}{2}}(1+L^2|z_L|)^{-\sigma(q)}\leq C L^{2-\frac{n}{q}}\|V\|_{L^q(\R^n)},
\end{align}
then $z:=\lim_{L\to\infty} z_L$ satisfies 
\begin{align}\label{second set limit}
    |z|^{\frac{1}{2}-\sigma(q)}\leq C\|V\|_{L^q(\R^n)}
\end{align}
with the same constant $C$ as in \eqref{second set expressed as inequality for zL}.
\end{lemma}

\begin{proof}
1. For $\epsilon>0$ to be chosen sufficiently small, we split the union of disks in \eqref{spectral inclusion scaled} into two sets (strictly speaking, two sequences of sets, depending on $L$): the first set $A_L$ contains all disks with $\lambda_k\leq L^{1+\epsilon}$, and the second set $B_L$ contains all disks with $\lambda_k> L^{1+\epsilon}$. The first set satisfies
\begin{align}
   A_L:= \bigcup_{\lambda_k\leq L^{1+\epsilon}}^{\infty}D(L^{-2}\lambda_k^2,CL^{-2}r_k)\subset [0,L^{2\epsilon}]+D(0,CL^{2(1+\epsilon)\sigma(q)-\frac{n}{q}}\|V\|_{L^q(\R^n)}),
\end{align}
where we used the embedding \eqref{torus embedding} to identify the $L^q$ norm of $V$ on $L\mathbb{T}^n$ with that on $\R^n$. Since $2\sigma(q)-n/q<0$ for all $q\in (n/2,\infty)$, we may choose $\eps$ so small (depending on $q$) that $2(1+\epsilon)\sigma(q)-\frac{n}{q}<0$. Hence, if there exists $L_0>1$ such that $z_L\in A_L$ for $L\geq L_0$, then
\begin{align}
    \lim_{L\to\infty} z_L\in  [0,\infty).
\end{align}
The alternative is that for all $L_0>1$, there exists $L\geq L_0$ such that $z_L\in B_L$. Then we can write $z_L=L^{-2}(\lambda_k^2+Cr_k\e^{\I\theta})$ for some $\theta\in [0,2\pi)$ and some $k$ satisfying $\lambda_k> L^{1+\epsilon}$. Therefore, we have
\begin{align}
    |z_L|\geq L^{-2}\lambda_k^2(1-Cr_k\lambda_k^{-2})\geq L^{2\eps}(1-o(1)).
\end{align}
This shows that $(z_L)$ does not converge, contradicting our assumption. Hence, the alternative cannot hold, and we are done. 

2. Fix $R>0$. If there is a subsequence, again denoted by $(z_L)$, such that $|z_L|\leq RL^{-2}$, then $\lim_{L\to\infty} z_L=0$, so \eqref{second set limit} holds trivially. Otherwise, $|z_L|> RL^{-2}$ for $L$ sufficiently large, and then \eqref{second set expressed as inequality for zL} implies
\begin{align}
    |z_L|^{\frac{1}{2}-\sigma(q)}\leq C(1+R^{-1})^{\sigma(q)} L^{1-\frac{n}{q}+2\sigma(q)}\|V\|_{L^q(\R^n)}.
\end{align}
Since $1-\frac{n}{q}+2\sigma(q)=0$ for $q\leq (n+1)/2$ and $R$ is arbitrarily large, the claim follows.
\end{proof}

\bibliographystyle{plain}
\bibliography{Bibliography.bib}

\begin{thebibliography}{10}

\bibitem{MR1819914}
A.~A. Abramov, A.~Aslanyan, and E.~B. Davies.
\newblock Bounds on complex eigenvalues and resonances.
\newblock {\em J. Phys. A}, 34(1):57--72, 2001.

\bibitem{MR4445914}
M.~D. Blair, X.~Huang, Y.~Sire, and C.~D. Sogge.
\newblock Uniform {S}obolev estimates on compact manifolds involving singular
  potentials.
\newblock {\em Rev. Mat. Iberoam.}, 38(4):1239--1286, 2022.

\bibitem{MR4561804}
S.~B\"{o}gli and J.-C. Cuenin.
\newblock Counterexample to the {L}aptev-{S}afronov conjecture.
\newblock {\em Comm. Math. Phys.}, 398(3):1349--1370, 2023.

\bibitem{MR3302640}
J.~Bourgain, P.~Shao, C.~D. Sogge, and X.~Yao.
\newblock On {$L^p$}-resolvent estimates and the density of eigenvalues for
  compact {R}iemannian manifolds.
\newblock {\em Comm. Math. Phys.}, 333(3):1483--1527, 2015.

\bibitem{MR3848231}
N.~Burq, D.~Dos Santos~Ferreira, and K.~Krupchyk.
\newblock From semiclassical {S}trichartz estimates to uniform {$L^p$}
  resolvent estimates on compact manifolds.
\newblock {\em Int. Math. Res. Not. IMRN}, pages 5178--5218, 2018.

\bibitem{MR3655948}
J.-C. Cuenin.
\newblock Sharp spectral estimates for the perturbed {L}andau {H}amiltonian
  with {$L^p$} potentials.
\newblock {\em Integral Equations Operator Theory}, 88(1):127--141, 2017.

\bibitem{MR4664427}
J.-C. Cuenin.
\newblock From spectral cluster to uniform resolvent estimates on compact
  manifolds.
\newblock {\em J. Funct. Anal.}, 286(2):Paper No. 110214, 42, 2024.

\bibitem{MR3200351}
D.~Dos Santos~Ferreira, C.~E. Kenig, and M.~Salo.
\newblock On {$L^p$} resolvent estimates for {L}aplace-{B}eltrami operators on
  compact manifolds.
\newblock {\em Forum Math.}, 26(3):815--849, 2014.

\bibitem{MR2820160}
R.~L. Frank.
\newblock Eigenvalue bounds for {S}chr\"{o}dinger operators with complex
  potentials.
\newblock {\em Bull. Lond. Math. Soc.}, 43(4):745--750, 2011.

\bibitem{MR3717979}
R.~L. Frank.
\newblock Eigenvalue bounds for {S}chr\"{o}dinger operators with complex
  potentials. {III}.
\newblock {\em Trans. Amer. Math. Soc.}, 370(1):219--240, 2018.

\bibitem{MR3682666}
R.~L. Frank and J.~Sabin.
\newblock Spectral cluster bounds for orthonormal systems and oscillatory
  integral operators in {S}chatten spaces.
\newblock {\em Adv. Math.}, 317:157--192, 2017.

\bibitem{MR3620715}
R.~L. Frank and L.~Schimmer.
\newblock Endpoint resolvent estimates for compact {R}iemannian manifolds.
\newblock {\em J. Funct. Anal.}, 272(9):3904--3918, 2017.

\bibitem{MR313856}
I.~C. Gohberg and E.~I. Sigal.
\newblock An operator generalization of the logarithmic residue theorem and
  {R}ouch\'{e}'s theorem.
\newblock {\em Mat. Sb. (N.S.)}, 84(126):607--629, 1971.

\bibitem{MR1335452}
Tosio Kato.
\newblock {\em Perturbation theory for linear operators}.
\newblock Classics in Mathematics. Springer-Verlag, Berlin, 1995.
\newblock Reprint of the 1980 edition.

\bibitem{MR894584}
C.~E. Kenig, A.~Ruiz, and C.~D. Sogge.
\newblock Uniform {S}obolev inequalities and unique continuation for second
  order constant coefficient differential operators.
\newblock {\em Duke Math. J.}, 55(2):329--347, 1987.

\bibitem{MR2540070}
A.~Laptev and O.~Safronov.
\newblock Eigenvalue estimates for {S}chr\"{o}dinger operators with complex
  potentials.
\newblock {\em Comm. Math. Phys.}, 292(1):29--54, 2009.

\bibitem{MR2366961}
Z.~Shen and P.~Zhao.
\newblock Uniform {S}obolev inequalities and absolute continuity of periodic
  operators.
\newblock {\em Trans. Amer. Math. Soc.}, 360(4):1741--1758, 2008.

\bibitem{MR1852334}
M.~A. Shubin.
\newblock {\em Pseudodifferential operators and spectral theory}.
\newblock Springer-Verlag, Berlin, second edition, 2001.
\newblock Translated from the 1978 Russian original by Stig I. Andersson.

\bibitem{MR835795}
C.~D. Sogge.
\newblock Oscillatory integrals and spherical harmonics.
\newblock {\em Duke Math. J.}, 53(1):43--65, 1986.

\bibitem{MR930395}
C.~D. Sogge.
\newblock Concerning the {$L^p$} norm of spectral clusters for second-order
  elliptic operators on compact manifolds.
\newblock {\em J. Funct. Anal.}, 77(1):123--138, 1988.

\bibitem{MR3186367}
C.~D. Sogge.
\newblock {\em Hangzhou lectures on eigenfunctions of the {L}aplacian}, volume
  188 of {\em Annals of Mathematics Studies}.
\newblock Princeton University Press, Princeton, NJ, 2014.

\bibitem{MR3645429}
C.~D. Sogge.
\newblock {\em Fourier integrals in classical analysis}, volume 210 of {\em
  Cambridge Tracts in Mathematics}.
\newblock Cambridge University Press, Cambridge, second edition, 2017.

\bibitem{MR482878}
A.~Weinstein.
\newblock Asymptotics of eigenvalue clusters for the {L}aplacian plus a
  potential.
\newblock {\em Duke Math. J.}, 44(4):883--892, 1977.

\end{thebibliography}
\end{document}